%% file: note_MP_law_0716.tex
\documentclass{article}

\input{rrp-macro-file}
\begin{document}

\title{Mar\v{c}enko-Pastur law for Tyler's M-estimator}

  \author{Teng Zhang\\
  \multicolumn{1}{p{.7\textwidth}}{\centering\emph{Department of Mathematics,\\ University of Central Florida\\  Orlando, Florida 32816, USA}\\ teng.zhang@ucf.edu }
  \\
  \\
  Xiuyuan Cheng\\\multicolumn{1}{p{.7\textwidth}}{\centering\emph{Applied Mathematics Program, \\Yale University,
New Haven, CT 06511, USA }\\ xiuyuan.cheng@yale.edu }\\
  \\
    Amit Singer\\\multicolumn{1}{p{.7\textwidth}}{\centering\emph{Department of Mathematics and PACM, Princeton University\\
Princeton, New Jersey 08544, USA}\\ amits@math.princeton.edu}
  }



\maketitle
\begin{abstract}
This paper studies the limiting behavior of Tyler's M-estimator for the scatter matrix, in the regime that the number of samples $n$ and their dimension $p$ both go to infinity, and $p/n$ converges to a constant $y$ with $0<y<1$. We prove that when the data samples $\bx_1, \ldots, \bx_n$ are identically and independently generated from the Gaussian distribution $\mathcal{N}(\b0,\bI)$, the operator norm of the difference between a properly scaled Tyler's M-estimator and $\sum_{i=1}^n\bx_i\bx_i^\top/n$ tends to zero. As a result, the spectral distribution of Tyler's M-estimator converges weakly to the Mar\v{c}enko-Pastur distribution. 
\end{abstract}

\section{Introduction}
%

Many statistical estimators and signal processing algorithms require the estimation of the covariance matrix of the data samples. When the underlying distribution of the data samples $\bx_1, \ldots, \bx_n\in\reals^{p}$ is assumed to have zero mean, a commonly used estimator is the sample covariance matrix $\bS_n=\sum_{i=1}^n\bx_i \bx_i^\top/n$.

However, the estimator $\bS_n$ is sensitive to outliers, and performs poorly in terms of statistical efficiency (i.e., it has a large variance) for heavy-tailed distributions, e.g., when the tail decays slower than the Gaussian tail.

A popular robust covariance estimator is an M-estimator introduced by Tyler~\cite{Tyler1987}, denoted by $\hat{\Sigma}$, which is the unique solution to
\begin{equation}\label{eq:tyler0}
\hat{\Sigma}=\frac{p}{n}\sum_{i=1}^n \frac{\bx_i\bx_i^\top}{\bx_i^\top\hat{\Sigma}^{-1}\bx_i},\,\,\tr(\hat{\Sigma})=1.
\end{equation}
Tyler's M-estimator gives the ``shape'' of the covariance, but is missing its magnitude. However, for many applications the ``shape'' of the covariance suffices, for example, the principal components can be obtained from the ``shape''.

Compared with the sample covariance estimator, Tyler's  M-estimator is more robust to heavy-tailed elliptical distributions. The density function of elliptical distributions in $\mathbb{R}^p$ takes the form
\[
f(\bx;\Sigma,\mu)=|\Sigma|^{-1/2}g\{(\bx-\mu)^\top\Sigma^{-1}(\bx-\mu)\},
\]
where $g$ is some nonnegative function such that $\int_0^\infty x^{p-1} g(x)\di x$ is finite. This family of distributions is a natural generalization of the Gaussian distribution by allowing heavier or lighter tails while maintaining the elliptical geometry of the equidensity contours. Elliptical distributions are considered important in portfolio theory and financial data, and we refer to the work by El Karoui~\cite[Section 4]{elkaroui2009} for further discussion. Besides, elliptical distributions are used by Ollila and Tyler~\cite{Ollila2012} in modeling radar data, where the  empirical distributions are heavy-tailed because of outliers.

Tyler~\cite{Tyler1987} showed that when a data set follows an unknown elliptical distribution (with mean zero), Tyler's M-estimator is the most robust covariance estimator in  the  sense  of  minimizing  the  maximum  asymptotic  variance. This property suggests that Tyler's M-estimator should be more accurate than the sample covariance estimator for elliptically distributed data. Empirically, it has been shown to outperform the sample covariance estimator in applications such as finance in the work by Frahm and Jaekel~\cite{Gabriel2007}, anomaly detection in wireless sensor networks by Chen et al.~\cite{Chen2011}, antenna array processing by Ollila and Koivune~\cite{Ollila2003}, and radar detection by Ollila and Tyler~\cite{Ollila2012}.

\subsection{Asymptotic analysis in a high-dimensional setting}
Many scientific domains customarily deal with sets of high dimensional data samples, and therefore it is increasingly common to work with data sets where the number of variables, $p$, is of the same order of magnitude as the number of observations, $n$. Under this high-dimensional setting, the asymptotic spectral properties of $\bS_n$ at the limit of infinite number of samples and infinite dimensions have been well studied by Johnstone~\cite{JohnstoneICM}. A noticeable example is the convergence of the spectral distribution. Denoting the eigenvalues of a matrix $\bA$ by $\lambda_1(\bA), \ldots, \lambda_n(\bA)$, its spectral
distribution is a discrete probability measure
\[
P=P(\cdot |\bA)=\frac{1}{n}\sum_{i=1}^n\delta_{\lambda_i(\bA)}
\]
with $\delta_{s}$ denoting Dirac measure at $s\in\reals$.   Mar\v{c}enko and Pastur~\cite{MP1967} showed that when the entries of $\{\bx_i\}_{i=1}^n$ are Gaussian independent identically distributed random variables with mean $0$ and variance $1$, $p,n\rightarrow \infty$ and $p/n\rightarrow y$, where $0< y\leq 1$, the spectral distribution of the eigenvalues of $\bS_n$ converges weakly to the Mar\v{c}enko-Pastur distribution defined by
\begin{equation}\label{eq:esd}
\rho_{\text{MP},y}(x) =\frac{1}{2\pi}\frac{y\sqrt{(y_+-x)(x-y_-)}}{x} \mathbf{1}_{[y_-,y_+]}, \,\,\,\text{where $y_{\pm}=(1\pm \sqrt{y})^2$.}
\end{equation}

Tyler's M-estimator is closely related to and can be considered as a special case of Maronna's M-estimator, which is defined by
\begin{equation}\label{eq:maronna1}
\bar{\Sigma}=\frac{1}{n}\sum_{i=1}^n u(\bx_i^\top\bar{\Sigma}^{-1}\bx_i)\bx_i\bx_i^\top
\end{equation}
for a nonnegative function $u:[0,\infty)\rightarrow [0,\infty)$. The properties of Maronna's M-estimator in the high-dimensional regime when $p,n\rightarrow\infty$, $p/n\rightarrow y$ and $0<y<1$ have been analyzed in recent works by  Couillet et al.~\cite{Couillet2013,Couillet2013_2}, which obtained convergence results for a properly scaled Maronna's M-estimator under the assumptions that $u(x)$ is nonnegative, nonincreasing and continuous; $xu(x)$ is nondecreasing and bounded and $\sup_{x}xu(x)>1$. Moreover, spiked random matrix models were also studied by Couillet~\cite{Couillet2015139}. However, these results do not apply to Tyler's M-estimator, although Frahm and Jaekel~\cite{Gabriel2007} have conjectured that the spectral distribution converges weakly to the Mar\v{c}enko-Pastur distribution. Some works focused on the performance of Tyler's M-estimator for the case $p,n\rightarrow\infty$ and $p/n\rightarrow 0$:  D\"{u}mbgen~\cite{Dumbgen1998} showed that the condition number of Tyler's estimator is $1+4\sqrt{p/n}+o(\sqrt{p/n})$, and Frahm and Glombek~\cite{Frahm2012} showed that the spectral distribution of $\sqrt{n/p}(\bar{\Sigma}-\bI)$ converges weakly to a semicircle distribution. 

\subsection{Main results}

In this paper, we analyze Tyler's M-estimator in the high-dimensional setting. Our main results,  Theorem~\ref{thm:norm} and Corollary~\ref{cor:spherical}, show that as $p, n\rightarrow\infty$ and $p/n\rightarrow y$, $0<y<1$, the spectral distribution of a properly scaled Tyler's M-estimator converges weakly to the Mar\v{c}enko-Pastur distribution $\rho_{\text{MP},y}(x)$. Based on the properties of Tyler's M-estimator, this paper analyzes the spectral distribution when data samples are i.i.d. drawn from other distributions, such as elliptical distributions.

When data samples are generated from elliptical distributions, the spectral distribution of the sample covariance estimator has been studied by El Karoui~\cite[Theorem 2]{elkaroui2009}. Compared to Corollary~\ref{cor:spherical}, the limiting spectral distribution of $\bS_n$ is much more complicated, and therefore our result might be more applicable in practice. 

High-dimensional analysis of Maronna's M-estimator of the covariance are generally obtained by showing that the operator norm of the difference between M-estimator and a standard Wishart matrix (or sample covariance matrix) tends to $0$: D\"{u}mbgen~\cite{Dumbgen1998} proved it by a linear expansion of the M-estimator, and Couillet et al.~\cite{Couillet2013,Couillet2013_2} proved it by representing Maronna's M-estimator as a weighted sum of $\bx_i\bx_i^\top$ and prove the uniform convergence of the weights. We follow the same  direction  while giving an alternate proof for the convergence of the weights, by considering the weights as the solution to an optimization problem, which can handle Tyler's M-estimator that is not covered by the results in Couillet et al.~\cite{Couillet2013,Couillet2013_2}. We remark that this approach can also be applied to Maronna's M-estimator to prove some of the results in Couillet et al.~\cite{Couillet2013,Couillet2013_2}.




The rest of the paper is organized as follows. In Section~\ref{sec:tyler} we introduce the representation of Tyler's M-estimator as a linear combination of $\bx_i\bx_i^\top$ and present the main result that when the data set is i.i.d. sampled from the  Gaussian distribution $\mathcal{N}(\b0,\bI)$, a properly scaled Tyler's is asymptotically equivalent to $\bS_n$ in the sense that $\|p\hat{\Sigma}-\bS_n\|\rightarrow 0$. As a result, the spectral distribution of Tyler's M-estimator converges weakly to the Mar\v{c}enko-Pastur distribution. We also extend the result to elliptical distributions. 
The technical proofs are given in Section~\ref{sec:proof}. While some Lemmas and technical proofs are also used by Couillet et al.~\cite{Couillet2013,Couillet2013_2} (for example, Lemma~\ref{lemma:deri} and the analysis in the proof of Theorem~\ref{thm:norm} are similar to Couillet et al. \cite[Lemma 2, Theorem 1]{Couillet2013} and Couillet et al. \cite[Lemma 6]{Couillet2013_2}, we still include them for the completeness of the paper.

As for notations, we use $c, c', C, C'$ to denote any fixed constants as $p,n\rightarrow\infty$ (though they may depend on $y$). Depending on the context, they might denote different values in different equations.

\section{Tyler's M-estimator in the High-dimensional Regime}\label{sec:tyler}
We introduce the representation of Tyler's M-estimator as a linear combination of $\bx_i\bx_i^\top$ in Section~\ref{sec:tyler_background}, and present the main result in Section~\ref{sec:tyler_main} that when the data set is i.i.d. sampled from the  Gaussian distribution $\mathcal{N}(\b0,\bI)$, $\|p\hat{\Sigma}-\bS_n\|$ converges to $0$ almost surely. Based on this observation, we prove that the spectral distribution of Tyler's M-estimator converges weakly the Mar\v{c}enko-Pastur distribution in Section~\ref{sec:tyler_general}. The generalization of the results to more general settings is also discussed in Section~\ref{sec:tyler_general}.

\subsection{Properties of Tyler's M-estimator}\label{sec:tyler_background}
The analysis for Tyler's M-estimator in this paper is based on the following representation, whose proof is deferred to  Section~\ref{sec:proof}. We remark that equation \eqref{eq:problem1} in Lemma~\ref{thm:alternate_tyler} has appeared in the work by Wiesel \cite[(27)]{Wiesel2012LSE} and Hardt and Moitra \cite[Section A]{Hardt2012} as ``covariance estimation in scaled Gaussian distributions'' and ``Barthe's convex program'', but its connection to Tyler's M-estimator has not been rigorously justified yet.
\begin{lemma}\label{thm:alternate_tyler}
Tyler's M-estimator can be written as
\begin{equation}\label{eq:equivalence_problem1}
\hat{\Sigma}=\sum_{i=1}^n\hat{w}_i\bx_i\bx_i^\top\Big/\tr\Big(\sum_{i=1}^n\hat{w}_i\bx_i\bx_i^\top\Big),
\end{equation}
where $\{\hat{w}_i\}_{i=1}^n$ are uniquely defined by
\begin{equation}\label{eq:problem1}
(\hat{w}_1, \ldots, \hat{w}_n)=\argmin_{w_i>0, \sum_{i=1}^n w_i=1} - \sum_{i=1}^n\ln w_i + \frac{n}{p} \ln\det\big(\sum_{i=1}^nw_i\bx_i\bx_i^\top\big).
\end{equation}
\end{lemma}
\subsection{Isotropic Gaussian Distribution}\label{sec:tyler_main}
In this subsection, we assume that $\{\bx_i\}_{i=1}^n\subset \reals^p$ are i.i.d. drawn from $\mathcal{N}(\b0,\bI)$. The main result, Theorem~\ref{thm:norm}, characterizes the convergence and convergence rate of Tyler's M-estimator to $\bS_n$ in terms of the operator norm. Its proof applies Lemma~\ref{thm:weights}, whose proof is rather technical and deferred to Section~\ref{sec:proof}.

Tyler's M-estimator does not exist when $p>n$ (see the argument by Zhang~\cite[Theorem III.1]{Zhang2012}) and it is not unique when $p=n$ (one may check that when $\bx_i=\be_i$ for all $1\leq i\leq p$, all diagonal matrices with trace $1$ satisfy \eqref{eq:tyler0}). As a result, throughout the paper we assume $y<1$.

\begin{lemma}\label{thm:weights}
If $\{\bx_i\}_{i=1}^n$  are i.i.d. sampled from $\mathcal{N}(\b0,\bI)$, then $\max_{1\leq i\leq n}|n\,\hat{w}_i - 1|$ converges to $0$ almost surely as $p, n\rightarrow\infty$: There exist $C,c, c'>0$ such that for any $\eps<c'$,
\begin{equation}\label{eq:weight_prob}
\Pr\left(\max_{1\leq i\leq n}|n\,\hat{w}_i - 1|\leq \eps\right)\geq 1-C ne^{-c\eps^2n}.\end{equation}
\end{lemma}

\begin{thm}\label{thm:norm}
Suppose that $\{\bx_i\}_{i=1}^n$ are i.i.d. sampled from $\mathcal{N}(\b0,\bI)$,  $p, n\rightarrow\infty$ and $p/n=y$, where $0<y<1$, then the operator norm of the difference between $\bS_n$ and a scaled Tyler's M-estimator converges to $0$ almost surely, and
there  exist $C,c,c'>0$ such that for any $\eps<c'$,
\begin{equation}\label{eq:norm_prob}
\Pr\left(\left\| p\, \hat{\Sigma} -\frac{1}{n} \sum_{i=1}^n \bx_i\bx_i^\top \right\|\leq \eps\right)\geq 1-C ne^{-c\eps^2n}.\end{equation}
\end{thm}

Theorem~\ref{thm:norm} implies that all first order properties of the sample covariance matrix extend to Tyler's estimator. The strategy of the proof for Theorem~\ref{thm:norm} is as follows. According to Lemma~\ref{thm:alternate_tyler}, a scaled Tyler's M-estimator is a linear combination of $\bx_i\bx_i^\top$, i.e., it can be written as $\sum_{i=1}^n \hat{w}_i\bx_i\bx_i^\top$ (up to a scaling). Then Lemma~\ref{thm:weights} shows that $n \hat{w}_i$ converges to $1$ uniformly. Based on the following matrix analysis, Theorem~\ref{thm:norm} is concluded.

\begin{proof}[Proof of Theorem~\ref{thm:norm}]
We first prove that for $\eps<c'$,
\begin{equation}\label{eq:operator norm1}
\Pr\left(\left\| \sum_{i=1}^n\hat{w}_i \bx_i\bx_i^\top -\frac{1}{n} \sum_{i=1}^n \bx_i\bx_i^\top \right\|\leq \eps\right)\geq 1-C ne^{-c\eps^2n}.
\end{equation}
Let $\bB_n=\sum_{i=1}^n(\hat{w}_i-1/n)\bx_i\bx_i^\top= \sum_{i=1}^n\hat{w}_i \bx_i\bx_i^\top - \sum_{i=1}^n \bx_i\bx_i^\top/n$, then \begin{align*}\|\bB_n\|=&\sup_{\|\bv\|=1}\bv^\top\bB_n\bv=
\sup_{\|\bv\|=1}\sum_{i=1}^n(\hat{w}_i-\frac{1}{n})(\bv^\top\bx_i)^2
\\\leq& \sup_{\|\bv\|=1}\sum_{i=1}^n\left\|\hat{\bw}-\frac{1}{n}\mathbf{1}\right\|_\infty(\bv^\top\bx_i)^2
\leq \|n\hat{\bw}-\mathbf{1}\|_\infty \|\frac{1}{n}\sum_{i=1}^n\bx_i\bx_i^\top \|.\end{align*}   Since $\|n\hat{\bw}-\mathbf{1}\|_\infty\rightarrow 0$ with probability estimated in \eqref{eq:weight_prob}, and  Davidson and Szarek~\cite[Theorem II.13]{Davidson2001} showed that $\|\sum_{i=1}^n\bx_i\bx_i^\top/n \|$ is bounded above by $(1+2\sqrt{y})^2$ with probability $1-C\exp(-cn)$, \eqref{eq:operator norm1} is proved.

Second, since \[\left\|\sum_{i=1}^n\hat{w}_i \bx_i\bx_i^\top\right\|\leq \left\| \sum_{i=1}^n\hat{w}_i \bx_i\bx_i^\top - \sum_{i=1}^n \bx_i\bx_i^\top/n \right\|+\left\|\sum_{i=1}^n \bx_i\bx_i^\top/n \right\|,
\]
\begin{equation}\text{
$\Pr\left(\| \sum_{i=1}^n\hat{w}_i \bx_i\bx_i^\top\|<C'\right)>1-Cn\exp(-cn)$.}\label{eq:norm_tyler}\end{equation} Besides,  $\tr(\sum_{i=1}^n\hat{w}_i\bx_i\bx_i^\top)=\sum_{i=1}^n \hat{w}_i\bx_i^\top\bx_i\rightarrow p$ in the same rate as in \eqref{eq:norm_tyler}: applying the concentration of high-dimensional Gaussian measure on the sphere by  Barvinok~\cite[Corollary 2.3]{Barvinok710}, we have
\begin{align}\label{eq:trace_convergence}
&\max\left[\Pr\left\{\sum_{i=1}^n \hat{w}_i\bx_i^\top\bx_i <p(1-\eps)\right\},\Pr\left\{\sum_{i=1}^n \hat{w}_i\bx_i^\top\bx_i > p/(1-\eps)\right\}\right]
\\\nonumber\leq &\max\left[\Pr\left\{\min_{1\leq i\leq n} \|\bx_i\|^2<p(1-\eps)\right\}, \Pr\left\{\max_{1\leq i\leq n} \|\bx_i\|^2>p/(1-\eps)\right\}\right]
<ne^{-\eps^2p/4}.
\end{align}
Combining \eqref{eq:norm_tyler}, \eqref{eq:trace_convergence} and \eqref{eq:equivalence_problem1},
\begin{equation}\label{eq:operator norm2}
\left\| \sum_{i=1}^n\hat{w}_i \bx_i\bx_i^\top -p\, \hat{\Sigma} \right\|= \left\| \sum_{i=1}^n\hat{w}_i \bx_i\bx_i^\top\right\| \left\{ 1- p/\tr(\sum_{i=1}^n\hat{w}_i \bx_i\bx_i^\top)\right\}
\end{equation}
converges in the same rate as specified in \eqref{eq:operator norm1}. \eqref{eq:norm_prob} is then proved by combining \eqref{eq:operator norm1}, \eqref{eq:operator norm2} and the triangle inequality.
\end{proof}

From the probabilistic estimation \eqref{eq:norm_prob} we obtain a convergence rate of $O(\sqrt{\ln n/n})$. The logarithmic factor is due to a ``max'' bound of $\{\hat{w}_i\}_{i=1}^n$ in Lemma~\ref{thm:weights}, while in fact, an ``average'' bound is expected. As a result, we conjecture that this $\sqrt{\ln n}$ factor could be possibly removed by a more rigorous argument.


\subsection{More general distributions and spectral distribution}\label{sec:tyler_general}
We remark that Theorem~\ref{thm:norm} can be extended from the setting of the normal distribution $\mathcal{N}(\b0,\bI)$ to any elliptical distribution $\mu_p$, which is characterized by the probability density function $\mu_p(\bx)=C(g_p)\det(\bT_p)^{-1/2}g_p(\bx^\top\bT_p^{-1}\bx)$, where $\bT_p$ is a positive definite matrix in $\reals^{p\times p}$, $g_p: [0,\infty)\rightarrow [0,\infty)$ satisfies $\int_{0}^\infty g_p(x)x^{p-1}<\infty$, and $C(g_p)$ is a normalization parameter that only depends on $g_p$. Then $\|\tr(\bT_p)\hat{\Sigma}-\sum_{i=1}^n \bx_i\bx_i^\top/n\|\rightarrow 0$ almost surely as $p,n\rightarrow\infty$. The analysis is based on Theorem~\ref{thm:norm}, the affine equivariance property of Tyler's M-estimator, and the fact that Tyler's M-estimator is unchanged if  $\{\bx_i\}_{i=1}^n$ are replaced by $\{c_i\bx_i\}_{i=1}^n$.

Another direction of generalization of Theorem~\ref{thm:norm} is the model by Couillet et al.~\cite{Couillet2013}: The elements of $\{\bx_i\}_{i=1}^n$ are i.i.d. sampled from an  either  real  or  circularly  symmetric  complex distribution with $\Expect(x_{ij})=0$, $\Expect(x_{ij}^2)=1$, and $\Expect(|x_{ij}|^{8+\eta})<\alpha$ for some $\eta, \alpha>0$. Then, following the proof in this paper (while replacing Lemma~\ref{lemma:deri} by~\cite[Lemma 2]{Couillet2013}), one can show that  $\left\| p\, \hat{\Sigma} - \sum_{i=1}^n \bx_i\bx_i^\top/n \right\|\rightarrow 0$ almost surely as $p,n\rightarrow\infty$.

We have the following results on the weak convergence of the spectral distribution of Tyler's M-estimator, where the first part proves the conjecture by Frahm and  Jaekel~\cite{Gabriel2007}. \begin{cor}\label{cor:spherical}
\begin{itemize}\item If $\{\bx_i\}_{i=1}^n$ are i.i.d. sampled from $\mathcal{N}(\b0,\bI)$, then the spectral measure $P(\cdot|p\hat{\Sigma})$ converges weakly to the Mar\v{c}enko-Pastur distribution.

\item If $\{\bx_i\}_{i=1}^n$ are i.i.d. sampled from an elliptical distribution $ C(g_p)g_p(\bx^\top\bT_p^{-1}\bx)$ such that the spectral measure of $\bT_p$ converges weakly to a distribution $H$ on $\reals$. Then the spectral measure   $P(\cdot|\tr(\bT_p)\hat{\Sigma})$ converges weakly to a probabilistic measure $\rho$ whose Stieltjes transform  $s(z)=\int 1/(x-z)\rho(\di x)$ ($z\in\mathbb{C}\setminus\mathbb{R}$) is given implicitly by \[s(z)=\int\frac{1}{t\{1-y-y z\,s(z)\}-z}\di H(t).\]
\end{itemize}
\end{cor}

This corollary can be proved by combining $\|\tr(\bT_p)\hat{\Sigma}-\sum_{i=1}^n \bx_i\bx_i^\top/n\|\rightarrow 0$, the analysis on the perturbation of eigenvalues by Bhatia~\cite[Corollary III.4.2]{Bhatia1997}, the spectral measure of $\sum_{i=1}^n \bx_i\bx_i^\top/n$ by Mar\v{c}enko and Pastur~\cite{MP1967}, Bai and Silverstein~\cite[(6.1.2)]{MP1967,bai2009spectral} and Slutsky's Lemma.

\section{Proof of Lemmas}\label{sec:proof}

\subsection{Proof of Lemma~\ref{thm:alternate_tyler}}\label{sec:thm_alternate_tyler}

%
%
%
We start with the definition
\begin{equation}\label{eq:problem2}
(\hat{z}_1, \ldots, \hat{z}_n)=\argmin_{\sum_{i=1}^n z_i=1} \ln\det\Big(\sum_{i=1}^ne^{z_i}\bx_i\bx_i^\top\Big)
\end{equation}
and
\begin{equation}\label{eq:problem3}
\hat{\Sigma}_z=\sum_{i=1}^ne^{\hat{z}_i}\bx_i\bx_i^\top.
\end{equation}

The solution to \eqref{eq:problem2} is unique, which follows from the convexity of the objective function (see Wiesel~\cite[Lemma 4]{Wiesel2012LSE}). Besides, noticing the equivalence between \eqref{eq:problem2} and \eqref{eq:problem1} (by plugging $w_i=  e^{z_i}/(\sum_{i=1}^ne^{z_i})$ and $z_i=\ln w_i - (\sum_{i=1}^n \ln w_i -1)/n$, there exists $c_1>0$ such that $\hat{\Sigma}_z=c_1\hat{\Sigma}$.

Next we will prove that $\hat{\Sigma}_z$ satisfies
 \begin{equation} \label{eq:tyler1}
\sum_{i=1}^n \frac{\bx_i\bx_i^\top}{\bx_i^\top\hat{\Sigma}_z^{-1}\bx_i}=c \hat{\Sigma}_z,\,\,\,\text{for some $c>0$}.
\end{equation}

By checking the directional derivative of the objective function in \eqref{eq:problem2}, for any $(\delta_1, \ldots, \delta_n)$ with $\sum_{i=1}^n\delta_i=0$,
\[
\sum_{i=1}^n \delta_i e^{\hat{z}_i}\bx_i^\top\hat{\Sigma}_z^{-1}\bx_i =0.
\]
Therefore, there exists $c_2$ such that
\begin{equation}\label{eq:Sigmaz}
e^{\hat{z}_i}\bx_i^\top\hat{\Sigma}_z^{-1}\bx_i = c_2,\,\,\text{for all $1\leq i\leq n$}.
\end{equation}
Therefore, \eqref{eq:tyler1} is proved by applying \eqref{eq:Sigmaz} and \eqref{eq:problem3}:
\[
\sum_{i=1}^n \frac{\bx_i\bx_i^\top}{\bx_i^\top\hat{\Sigma}_z^{-1}\bx_i}
=\sum_{i=1}^n e^{\hat{z}_i}{\bx_i\bx_i^\top}/c_2
=\hat{\Sigma}_z/c_2,
\]

Since $\hat{\Sigma}_z=c_1\hat{\Sigma}$, \eqref{eq:tyler1} also holds when $\hat{\Sigma}_z$ is replaced by $\hat{\Sigma}$:
\begin{equation}\label{eq:tyler2}
\sum_{i=1}^n \frac{\bx_i\bx_i^\top}{\bx_i^\top\hat{\Sigma}^{-1}\bx_i}=c \hat{\Sigma},\,\,\,\text{for some $c>0$}.
\end{equation}

At last, we will prove that $\hat{\Sigma}$ satisfies the definition of Tyler's M-estimator in \eqref{eq:tyler0}, that is, the constant $c$ in \eqref{eq:tyler2} is given by $c=n/p$. For the objective function
\[
F(\Sigma)=\sum_{i=1}^n\ln(\bx_i^\top\Sigma^{-1}\bx_i)+c\ln\det(\Sigma),
\]
its derivative with respect to $\Sigma^{-1}$ is given by
\[
\sum_{i=1}^n\bx_i^\top(\bx_i^\top\Sigma^{-1}\bx_i)^{-1}\bx_i-c\Sigma.
\]
Therefore, $\hat{\Sigma}$ is a stationary point of $F(\Sigma)$. Since $F(\Sigma)$ is geodesically convex (argument follows directly from Wiesel~\cite{Wiesel2012LSE} and Zhang~\cite{Zhang2012}), $\hat{\Sigma}$ is the global minimizer of $F(\Sigma)$.

However, the minimizer of $F(\Sigma)$ exists only when $c=n/p$. Since $F(a\bI)=\sum_{i=1}^n \ln(\bx_i^\top\bx_i)-n\ln a+c\,p\ln a$, we have
\[
F(a\bI)\rightarrow -\infty\,\,\,\,\begin{cases} &\text{as $a\rightarrow 0$, if $c>n/p$}\\&\text{as $a\rightarrow \infty$, if $c<n/p$}\end{cases}.
\]

Therefore, the constant $c$ in \eqref{eq:tyler2} is given by $c=n/p$, and Lemma~\ref{thm:alternate_tyler} is proved.

\subsection{Proof of Lemma~\ref{thm:weights}}
We start with an outline of the proof, which consists of three parts. First, we rewrite the constrained optimization problem \eqref{eq:problem1} to the problem of finding the root of $g(\bw)$, which will be defined in \eqref{eq:definition_g}. Since the root of $g(\bw)$ is $n\hat{\bw}-1$, we only need to show the convergence of the root of $g(\bw)$. Second, we will show that $g(\b0)$ converges to $\b0$, $\grad g(\b0)$ is large and the variation of $\grad g(\bw)$ is bounded. Finally, we will use a perturbation analysis and the observations on $g(\b0)$ and $\grad g(\bw)$ to show that the root of $g(\bw)$ converges to $\b0$.

The proof depends on Lemma~\ref{lemma:deri}, Lemma~\ref{lemma:hessian} and Lemma~\ref{lemma:pertubation}, and their proofs are postponed to subsequent sections.


\begin{lemma}\label{lemma:pertubation}
%
For a function $f(\bw): \reals^p\rightarrow\reals^p$, assume that $\grad f(\mathbf{\b0})=\bI$, and $\|\grad f(\bw)-\grad f(\b0)\|_{\infty}=\max_{i\leq i\leq p}\|\grad f_i(\bw)-\grad f_i(\b0)\|_{\infty}< C_5 \|\bw\|_\infty$ for $\|\bw\|_\infty\leq 1$, and $\|f(\b0)\|_{\infty}< \min(1/9C_5,1/3)$. Then there exists $\tilde{\bw}$ such that $\|\tilde{\bw}\|_{\infty}<3 \|f(\b0)\|_{\infty}$ and $ f(\tilde{\bw})=\mathbf{0}$.
\end{lemma}

\begin{lemma}\label{lemma:deri}
If $\bx_i\sim \mathcal{N}(\b0,\bI)$ for all $1\leq i\leq n$, and $\bS=\sum_{i=1}^n\bx_i\bx_i^\top/n$, then there exists $c,C,c'>0$ such that for any $\eps<c'$,
 \[\Pr\left(\max_{1\leq i\leq n}|\frac{1}{p}\bx_i^\top\bS^{-1}\bx_i-1|<\eps\right)\geq 1-C ne^{-c\eps^2n}.\]
\end{lemma}

\begin{lemma}\label{lemma:hessian}For the $n\times n$ matrix $\bA$ defined by $\bA_{ij}=(\bx_i^\top\bS^{-1}\bx_j)^2/(n\,p)$, (a)
$\|\bA\|_{\infty}<2$ with probability $1-Cn\exp(-cn)$.

(b) There exists $c=c(p,n)>0$ and $C_2=C_2(y)>0$  such that $\|(\bI-\bA+c \mathbf{1}\mathbf{1}^\top)^{-1}\|_{\infty}<C_2$ with probability $1-Cn\exp(-cn)$.
\end{lemma}

We start the first part of the proof with the construction of $g(\bw)$. We let \begin{equation}\label{eq:definition_g}
g(\bw)=\grad G(\bw + \mathbf{1}),
\end{equation}
where
\begin{equation}\label{eq:unconstrained}
G(\bw)= - \sum_{i=1}^n\ln {w}_i + \frac{n}{p} \ln\det(\sum_{i=1}^n{w}_i\bx_i\bx_i^\top)+ \frac{c_0}{2}(\sum_{i=1}^n w_i-n)^2,
\end{equation}
and the constant $c_0$ will be specified later before \eqref{eq:A_i_row}.

It is easy to prove that the minimizer of $G(\bw)$ and the zeros of $\grad G(\bw)$ must satisfy $\sum_{i=1}^n w_i = n$ (otherwise $n\bw/(\sum_{i=1}^n w_i)$ is a better minimizer and $\grad G(\bw)$ is nonzero). Therefore minimizing \eqref{eq:unconstrained} is equivalent to minimizing $- \sum_{i=1}^n\ln {w}_i + n/p\cdot \ln\det(\sum_{i=1}^n{w}_i\bx_i\bx_i^\top)$ with constraint $\sum_{i=1}^n w_i = n$, which is the same as \eqref{eq:problem1} except for the constraint. Noticing that a scaling of $\bw$ increases $- \sum_{i=1}^n\ln {w}_i + n/p\cdot \ln\det(\sum_{i=1}^n{w}_i\bx_i\bx_i^\top)$ by a constant only depending on the scale, the minimizer of \eqref{eq:unconstrained} is unique and it is $n\hat{\bw}$, where $\hat{\bw}$ is defined in \eqref{eq:problem1}. By the convexity of its equivalent problem \eqref{eq:problem2}, the root of $g(\bw)$ is also unique and it is $n\hat{\bw}-1$. 

For the second part of the proof, we start by proving that $g(\b0)$ is small. By calculation, the $i$-th component of function $g(\bw)$ is
\[
g_i(\bw)= - \frac{1}{w_i+1} + \frac{n}{p}\bx_i^\top (n\bS+\sum_{i=1}^n{w}_i\bx_i\bx_i^\top)^{-1}\bx_i +c_0\sum_{i=1}^n w_i.
\]
Applying Lemma~\ref{lemma:deri},
\begin{equation}\label{eq:condition1}\text{
$\Pr\left(\|g(\b0)\|_\infty<\eps\right)\geq 1-C ne^{-c\eps^2n}$.
}\end{equation}
Now we will prove that $\grad g(\b0)$ is bounded from below. By calculation, its $(i,j)$-th entry is \[
\big\{\grad g(\bw)\big\}_{i,j}=I(i=j)\frac{1}{(w_i+1)^2} - \frac{n}{p} \Big\{\bx_i^\top\Big(n\bS+\sum_{i=1}^nw_i\bx_i\bx_i^\top\Big)^{-1}\bx_j\Big\}^2 +c_0.
\]
Applying Lemma~\ref{lemma:hessian}, \begin{equation}\label{eq:condition2}\text{$\|\{\grad g(\b0)\}^{-1}\|_\infty<C_2$ with probability $1-Cne^{-cn}$.}\end{equation}

Now we bound the variation of $\grad g(\bw)$
 in the region $\|\bw\|_\infty<1/2$. Apply $|1/(w_i+1)^2-1|< 3 |w_i-1|\leq 3\|\bw\|_\infty$ and coordinatewise comparison,
\[
|\grad_{i,j} g(\bw)-\grad_{i,j} g(\b0)|\leq I(i=j)\left(3\|\bw\|_\infty\right)+3\|\bw\|_\infty\cdot \frac{n}{p}|\bA_{ij}|.
\]
Therefore, the variation of $\grad g(\bw)$ is bounded by \begin{equation}\label{eq:condition3}\text{$\|\grad g(\bw)-\grad g(\b0)\|_\infty < (3+3n\|\bA\|_\infty/p) \|\bw\|_\infty$.}\end{equation}

At last we finish the third part of the proof of Lemma~\ref{thm:weights} by applying Lemma~\ref{lemma:pertubation} to $f(\bw)=\{\grad g(\b0)\}^{-1}g(\bw/2)$. It is easy to verify that $\grad f(\b0)=\bI$. Due to \eqref{eq:condition1} and \eqref{eq:condition2}, $\|f(\b0)\|_{\infty}\leq \|(\grad g(\b0))^{-1}\|_\infty \|g(\b0)\|_\infty\rightarrow 0$ in the same rate as in \eqref{eq:condition1} and  $\|f(\b0)\|_{\infty}< \min(1/9C_5,1/3)$ holds with probability $1-Cne^{-cn}$.  Due to \eqref{eq:condition1}, \eqref{eq:condition3}, and the boundedness of $\|\bA\|_\infty$ (Lemma~\ref{lemma:hessian}), $\|\grad f(\bw)-\grad (\b0)\|_{\infty}< C_5 \|\bw\|_\infty$ also holds with probability $1-Cne^{-cn}$. Therefore the assumption in Lemma~\ref{lemma:pertubation} holds with probability $1-Cne^{-cn}$ and there exists $\tilde{\bw}$ such that $f(\tilde{\bw})=0$ and \begin{equation}\|\tilde{\bw}\|_{\infty}<3 \|f(\b0)\|_{\infty}.\label{eq:tildebw}\end{equation}

When $f(\tilde{\bw})=0$, we have $g(2\tilde{\bw})=0$ and by previous discussion $2\tilde{\bw}=n\hat{\bw}-1$. therefore \eqref{eq:tildebw} gives
\[
\|n\hat{\bw}-1\|_{\infty}<6 \|f(\b0)\|_{\infty}.
\]

Since $\|f(\b0)\|_{\infty}$ converges to $0$ in the rate as in \eqref{eq:condition1}, $\|n\hat{\bw}-1\|_{\infty}$ converges in the same rate and Lemma~\ref{thm:weights} is proved.
\subsubsection{Proof of Lemma~\ref{lemma:pertubation}}
\begin{proof}
When $\|\bw\|_\infty\leq 1$,
\begin{align}\label{eq:derivative_diff}
&f_j(\bw)-f_j(\mathbf{0}) =\int_{t=0}^{1}\left\langle\be_j \bw^\top,\grad f(t\,\bw)\right\rangle\di t\\\nonumber
=&\int_{t=0}^{1}\left\langle\be_j \bw^\top,\grad f(t\,\bw)-\grad f(\b0) +\bI\right\rangle\di t= w_j+ \int_{t=0}^{1}\bw^\top \{\grad f(t\,\bw)-\grad f(\b0)\}\be_j\di t
\\\leq& w_j+ \|\int_{t=0}^{1}\bw^\top \{\grad f(t\,\bw)-\grad f(\b0)\}\|_\infty\leq w_j + C_5 \|\bw\|_\infty^2.\nonumber
\end{align}
Similarly
\begin{equation}\label{eq:derivative_diff2}
f_j(\bw)-f_j(\mathbf{0})
\geq -C_5 \|\bw\|_\infty^2 + w_j.
\end{equation}


To prove it, we consider the continuous mapping $h(\bw)=\bw - f(\bw)/(4+9C_5)$ and will prove that $h$ maps $\calA$ to itself, where
\[
\calA=\{\bw: \bw\in [-3\eta,3\eta]^n\}\,\,\text{and  $\eta=\|f(\mathbf{0})\|_{\infty}$}.
\]

1. $|w_i|<2\eta$. Then apply \eqref{eq:derivative_diff} and \eqref{eq:derivative_diff2} (they are applicable since for any $\bw\in\calA$, $\|\bw\|_\infty\leq 1$), we have $|f_i(\bw)|< |f_i(\mathbf{0})|+ C_5 \|\bw\|_\infty^2 + |w_i|\leq \eta + C_5 (3\eta)^2 +3\eta <(4+9C_5) \eta$ ($\eta^2<\eta$ since $\eta<1$). Therefore, $|h_i(\bw)|\leq |w_i| + |f_i(\bw)|/(4+9C_5)\leq 3\eta$.

2. $w_i>2\eta$, then applying \eqref{eq:derivative_diff2},
\[
f_i(\bw)\geq -|f_i(\b0)|+w_i-C_5\|\bw\|_\infty^2\geq -\eta + 2\eta -C_5 (3\eta)^2.
\]
Since $\eta<1/9C_5$, we have $f_i(\bw)<0$ and therefore $h_i(\bw)\leq w_i\leq 3\eta$.

Similar to case 1 we can prove that $h_i(\bw)\geq  -3\eta$. Therefore $|h_i(\bw)|<3\eta$.

3. Similar to case 2, when $w_i<-2\eta$,  $|h_i(\bw)|<3\eta$.

Therefore the continuous mapping $h$ maps the convex, compact set $\calA$ to itself. By Schauder fixed point theorem, $h(\bx)$ has a fixed point in $\calA$ and Lemma~\ref{lemma:pertubation} is proved with $\tilde{\bw}$ being the fixed point.
\end{proof}
\subsubsection{Proof of Lemma~\ref{lemma:deri}}
Assuming the SVD decomposition of $\bX$ is $\bX=\bU\Sigma\bV^\top$, where $\bU\in\reals^{n\times p}$ and $\bU^\top\bU=\bI$. Since $\bx_i\sim \mathcal{N}(\b0,\bI)$ for all $1\leq i\leq n$, $\bU$ is uniformly distributed over the space of all orthogonal $n\times p$ matrices. Since \begin{equation}\label{eq:bU}\bX\bS^{-1}\bX=(\bU\Sigma\bV^\top)(\frac{1}{n}\bV\Sigma^2\bV^\top)^{-1}(\bU\Sigma\bV^\top),
\end{equation}
if we write the row of $\bU$ by $\bu_1, \ldots, \bu_n$, then $\frac{1}{n}\bx_i\bS^{-1}\bx_i=\bu_i^\top\bu_i=\|\bu_i\|^2$.

Since $\bU$ can be considered as the first $p$ columns of a random $n\times n$ orthogonal matrix (with haar measure over the set of all $n\times n$ orthogonal matrices), $\bu_i$ can be considered as the first $p$ entries from a random vector of length $n$ that is sampled from the uniform sphere in $\reals^n$.

Therefore, $\|\bu_i\|^2\sim \sum_{j=1}^p g_j^2/\sum_{j=1}^n g_j^2$ for i.i.d. random variables $\{g_j\}_{j=1}^n\sim \mathcal{N}(0,1)$. Applying the the concentration result by Barvinok~\cite[Corollary 2.3]{Barvinok710}, we have
\begin{equation}\label{eq:concentration1}
\Pr\left\{\sum_{i=1}^ng_i^2\geq \frac{n}{1-\eps}\right\}\leq e^{-\eps^2n/4}
\end{equation}
and
\begin{equation}\label{eq:concentration2}
\Pr\left\{\sum_{i=1}^ng_i^2\leq {n}(1-\eps)\right\}\leq e^{-\eps^2n/4},
\end{equation}
therefore
\begin{align*}
&\Pr\left\{ \frac{p(1-\eps)^2}{n}\leq \|\bu_1\|^2\leq \frac{p}{n(1-\eps)^2}\right\}
\geq \Pr\left\{{p}(1-\eps)\leq \sum_{i=1}^p g_i^2\leq \frac{p}{1-\eps}\right\}
\\&+\Pr\left\{{n}(1-\eps)\leq \sum_{i=1}^n g_i^2\leq \frac{n}{1-\eps}\right\}
\geq 1- 2e^{-\eps^2p/4}-2e^{-\eps^2n/4}.
\end{align*}

For $\eps\leq 0.1$, we have
\begin{align}
&\Pr\left\{\max_{1\leq i\leq n}|\frac{1}{p}\bx_i^\top\bS^{-1}\bx_i-1|\leq \eps\right\}
\geq 1- n \Pr\left\{|\|\bu_1\|^2-\frac{p}{n}|>  \frac{p}{n}\eps\right\}
\nonumber\\\geq &1-n\left[1-\Pr\left\{ \frac{p(1-\eps/3)^2}{n}\leq \|\bu_1\|^2\leq \frac{p}{n(1-\eps/3)^2}\right\}\right]
\\\geq& 1-2ne^{-\eps^2p/36}-2ne^{-\eps^2n/36},\label{eq:deri1}
\end{align}
where the second inequality follows from $1-3\eps\leq (1-\eps)^2$ and $1/{(1-\eps)^2}\leq 1+3\eps$.

\subsubsection{Proof of Lemma~\ref{lemma:hessian}}

(a) Since $\|\bA\|_\infty=\max_{1\leq i\leq n} (\sum_{1\leq j\leq n} \bA_{ij})$, and
\begin{align}
&\sum_{1\leq j\leq n} \bA_{ij}=
\sum_{1\leq j\leq n} \frac{1}{np}\bx_i^\top\bS^{-1}\bx_j \bx_j^\top\bS^{-1}\bx_i
= \bx_i^\top\bS^{-1} (\sum_{1\leq j\leq n} \bx_j\bx_j^\top)\bS^{-1}\bx_i/np
\\=&\bx_i^\top\bS^{-1} (n\bS)\bS^{-1}\bx_i/np =\bx_i^\top\bS^{-1}\bx_i/p,\label{eq:A_infinity}
\end{align}
it follows from \eqref{eq:deri1} with $\eps=0.1$ that $\|\bA\|_{\infty}<2$ holds with probability $1-Cn\exp(-cn)$.

(b) We first prove that there exists $C_3=C_3(y)$ such that \begin{equation}\|\bA-c_0 \mathbf{1}\mathbf{1}^\top\|_{\infty}\leq C_3 < 1\,\,\,\,\text{with probability $1-Cn\exp(-cn)$.}\label{eq:proof3}\end{equation}


We start with the proof of \eqref{eq:proof3} with another lemma:
\begin{lemma}\label{lemma:proof31}
There exists a $c_4>0$ such that with probability $1-C\exp(-cn)$,
\[\sum_{j=1}^n I(\bx_1^\top\bx_j>c_4\sqrt{p}) > 0.75 n.\]
\end{lemma}

Davidson and S. Szarek~\cite[Theorem II.13]{Davidson2001} showed that There exists $C_4=C_4(y)$ such that $\|\bS\|<C_4$ with probability $1-Cn\exp(-cn)$. Therefore $\bx_i^\top\bS^{-1}\bx_j\geq  \bx_i^\top\bx_j/C_4 $ and Lemma~\ref{lemma:proof31} implies that for any $1\leq i\leq n$:
\begin{equation}\label{eq:A_probability}\sum_{j=1}^n I(\bx_i^\top\bS^{-1}\bx_j>c_4\sqrt{p}/C_4) > 0.75\,\,\,\,\text{with probability $1-C\exp(-cn)$}.\end{equation}

Let $c_0=(c_4/C_4)^2/n$, then \eqref{eq:A_probability} implies %
%
\begin{equation}\label{eq:A_i_row}
\sum_{1\leq j\leq n} |\bA_{i,j}-c|\leq \sum_{1\leq j\leq n} |\bA_{i,j}|   - 0.25 c \, n \leq   \bx_i^\top\bS^{-1}\bx_i/p- 0.25 (c_4/C_4)^2,
\end{equation}
where the last step follows from \eqref{eq:A_infinity}.

Applying the estimation of $\bx_i^\top\bS^{-1}\bx_i/p$ in \eqref{eq:deri1} and a union bound argument over all $1\leq i\leq n$ to \eqref{eq:A_i_row}, \eqref{eq:proof3} is proved for $C_3=1+\eta-0.25 (c_4/C_4)^2$.

Lemma~\ref{lemma:hessian}(b) follows from~\eqref{eq:proof3} with $C_2=1/({1-C_3})$, where the expansion of $(\bI-\bA+c \mathbf{1}\mathbf{1}^\top)^{-1}$ exists since $\|\bA+c \mathbf{1}\mathbf{1}^\top\|\leq \|\bA+c \mathbf{1}\mathbf{1}^\top\|_\infty<1$. Applying $\|\bB_1\bB_2\|_\infty\leq \|\bB_1\|_\infty\|\bB_2\|_\infty$, we have
\begin{equation}
\|(\bI-\bA+c \mathbf{1}\mathbf{1}^\top)^{-1}\|_{\infty}
=\|\sum_{k=0}^\infty (c \mathbf{1}\mathbf{1}^\top-\bA)^k\|_{\infty}
\leq  \sum_{k=0}^\infty \|c \mathbf{1}\mathbf{1}^\top-\bA\|_\infty^k \leq \sum_{k=0}^\infty C_3^i=\frac{1}{1-C_3}.
\end{equation}

%
%
%

\subsubsection{Proof of Lemma~\ref{lemma:proof31}}
We first show that there exists $c_4$ such that for all $p$,
\begin{equation}\label{eq:lemma31_expect}
{\rm E}\{I(|\bx_1^\top\bx_2|>c_4\sqrt{p})\}\geq 0.85.
\end{equation}
WLOG we rotate $\bx_1$ such that it is nonzero only at the first coordinate, and $\bx_2=(g_1,g_2,...,g_p)$ where $g_i\sim \mathcal{N}(0,1)$. Then $|\bx_1^\top\bx_2|=|g_1|\,\|\bx_1\|$.

Notice that $\|\bx_1\|^2$ is the sum of $p$ independent $\chi_1^2$ distribution and $E \chi_1^2 =1$, by central limit theorem, $\|\bx_1\|\leq \sqrt{2p}$ with probability $1-Ce^{-cn}$. Besides, $\Pr(|g_1|>\sqrt{2}\, c_4)\geq 0.85$ for $c_4=\Phi^{-1}(1-0.85/2)/\sqrt{2}$. Therefore \eqref{eq:lemma31_expect} is proved by combining the estimations on $|g_1|$, $\bx_1$ and $|\bx_1^\top\bx_2|=|g_1|\,\|\bx_1\|$.

To obtain Lemma~\ref{lemma:proof31} from \eqref{eq:lemma31_expect}, we apply Hoeffding's inequality to the indicator function $I(|\bx_i^\top\bx_j|>c_4\sqrt{p})$ over all $1\leq j\leq n, j\neq i$.


\section{Summary}
We showed that Tyler's M-estimator is asymptotically equivalent to $\bS_n$ in the sense that $\|p\hat{\Sigma}-\bS_n\|\rightarrow 0$ as $p,n\rightarrow\infty$ and $p/n\rightarrow y$, where $0<y<1$ and data samples follow the distribution of $\mathcal{N}(\b0,\bI)$.  We also proved the conjecture that the spectral distribution of Tyler's M-estimator converges weakly to the Mar\v{c}enko-Pastur distribution, and extended the results to elliptical distributions.

There are several possible future directions of this work. First, it would be interesting to analyze the second order statistics of Tyler's M-estimator, considering that Couillet~\cite{Couillet2016249} has already investigated Maronna's M-estimators. Second, we would like to theoretically quantify the behavior of Tyler's M-estimator in the spiked covariance model by Couillet~\cite{Couillet2015139}, which includes the analysis of the distribution of the top eigenvalue for the null cases and the analysis of the non-null case. A recent work by Morales-Jimenez et al.~\cite{Morales2015} on the non-null case introduced a mixture model that consists of a Gaussian distribution and some deterministic or random outliers, and analyzed the performance of Maronna's M-estimator. Analyzing the performance of Tyler's M-estimator in this model would be another possible future direction.

\section*{Acknowledgements}
A. Singer was partially supported by Award Number FA9550-12-1-0317 and FA9550-13-1-0076 from AFOSR, by Award Number R01GM090200 from the NIGMS, and by Award Number LTR DTD 06-05-2012 from the Simons Foundation.

\bibliographystyle{abbrv}
\bibliography{bib-rrp}
\end{document}

%% file: rrp-macro-file.tex
\usepackage{amsthm}
\usepackage{amsmath}
\usepackage{amssymb}
\usepackage{bm}
\usepackage{mathrsfs}
\usepackage[dvips]{graphicx}

\usepackage{algorithmic}
\usepackage{algorithm}

\usepackage{color}

\usepackage{url}

\usepackage{supertabular}

\newtheorem{thm}{Theorem}
\newtheorem{cor}[thm]{Corollary}
\newtheorem{lemma}[thm]{Lemma}

\theoremstyle{definition}

\theoremstyle{remark}

\numberwithin{thm}{section}

\DeclareMathAlphabet{\mathsfsl}{OT1}{cmss}{m}{sl}

\renewcommand{\phi}{\varphi}

\newcommand{\eps}{\varepsilon}

\newcommand{\grad}{\nabla}

\newcommand{\argmin}{\operatorname*{arg\; min}}

\newcommand{\Expect}{\operatorname{\mathbb{E}}}

\newcommand{\bx}{\boldsymbol{x}}

\newcommand{\bX}{\boldsymbol{X}}

\newcommand{\bU}{\boldsymbol{U}}
\newcommand{\bV}{\boldsymbol{V}}

\def\reals{\mathbb{R}}
\def\bx{\boldsymbol{x}}
\def\b0{\mathbf{0}}

\def\bU{\boldsymbol{U}}
\def\bu{\boldsymbol{u}}
\def\be{\boldsymbol{e}}

\def\bv{\boldsymbol{v}}
\def\bA{\boldsymbol{A}}
\def\bT{\boldsymbol{T}}
\def\bB{\boldsymbol{B}}
\def\bS{\boldsymbol{S}}

\def\bw{\boldsymbol{w}}
\def\bI{\mathbf{I}}
\def\b1{\mathbf{1}}
\def\b0{\mathbf{0}}
\def\calA{\mathcal{A}}

\def\tr{\mathrm{tr}}

\usepackage{algorithmic}
\usepackage{algorithm}
\newcommand{\di}{{\,\mathrm{d}}}